\documentclass{amsart}

\usepackage{amssymb}
\usepackage{amsthm}
\usepackage{amsmath}
\usepackage{enumerate}
\usepackage{mathrsfs}
\usepackage{graphicx}
\usepackage{bookmark}
\usepackage{tikz}
\usetikzlibrary{decorations.markings}

\usepackage{floatrow}

\usepackage{hyperref}

\theoremstyle{theorem}
\newtheorem{theorem}{Theorem}[section]
\newtheorem{proposition}[theorem]{Proposition}
\newtheorem{lemma}[theorem]{Lemma}
\newtheorem*{lemma*}{Lemma}

\theoremstyle{definition}
\newtheorem{definition}[theorem]{Definition}
\newtheorem{example}[theorem]{Example}

\theoremstyle{remark}
\newtheorem{remark}[theorem]{Remark}

\numberwithin{equation}{section}

\newcommand{\B}{\mathcal{B}}

\newcommand{\Z}{\mathbb{Z}}
\newcommand{\R}{\mathbb{R}}
\newcommand{\N}{\mathbb{N}}

\newcommand\longmapsfrom{\mathrel{\reflectbox{\ensuremath{\longmapsto}}}}

\DeclareMathOperator{\spn}{span}

\DeclareMathOperator{\ob}{ob}

\DeclareMathOperator{\diag}{diag}

\usepackage{graphicx}
\usepackage{tikz}
\usetikzlibrary{decorations.markings}

\usepackage{hyperref}

\usepackage{floatrow}

\begin{document}
	
	\title[A categorical approach to operator semigroups]{A categorical approach to operator semigroups}
	
	\author{Abraham C.S.\ Ng}
	\thanks{Address: School of Mathematics and Applied Statistics, University of Wollongong, Australia}
	\thanks{Email: abraham\_ng@uow.edu.au, ORCID iD: 0000-0002-1701-7904}

	\date{\today}

	
	
	\begin{abstract}
		The aim of this paper is to exploit the structure of strongly continuous operator semigroups in order to formulate a categorical framework in which a fresh perspective can be applied to past operator theoretic results. In particular, we investigate the inverse-producing Arens extension for Banach algebras (Trans.\ Amer.\ Math.\ Soc.\  88:536-548, 1958) adapted for operators and operator semigroups by Batty and Geyer (J.\ Operator Theory 78(2):473-500, 2017) in this new framework, asking and answering questions using categorical language. We demonstrate that the Arens extension defines an extension functor in this setting and that it forms an adjunction with the suitably defined forgetful functor. As a by-product of this categorical framework, we also revisit the work on Banach direct sums by Lachowicz and Moszy\'nski (Semigroup Forum 93(1):34-70, 2016). This paper can be considered as a brief exploration of the triple interface between operator semigroups, Banach algebras, and category theory.
	\end{abstract}
	
	\maketitle

\subsection*{Acknowledgements}

The author is grateful to many people for their help during the production of this article. To Charles Batty and David Seifert for insightful discussions on the subject of this article, to Joshua Ciappara for making some crucial suggestions from an algebraist's perspective (such as suggesting that I look at adjunctions), and to the reviewer for numerous useful comments as a result of which this paper is much improved. The author is also grateful to the University of Sydney for partially funding the majority of this work through the Barker Graduate Scholarship during his time as a doctoral student at the University of Oxford.

\subsection*{Declarations}

\textit{Funding:} This work was partially supported by the University of Sydney through the Barker Graduate Scholarship. The revisions were completed when the author was funded by ARC grant DP180100595. \textit{Conflicts of interest/Competing interests:} Not applicable. \textit{Availability of data and material:} Not applicable. \textit{Code availability:} Not applicable.

\section{Introduction}\label{sec:1}

The framework of category theory provides a powerful tool with which to view certain mathematical objects that have intrinsic structure in a holistic way. There are many ways in which category theory appears in functional analysis with various treatments appearing in the literature. A survey detailing the appearance of category theory within Banach space theory can be found in \cite{Ca} while $C^*$-categories are treated in various places, including for the first time in \cite{GLR}. Interpolation theory and representations of groups, natural candidates for a categorical approach, are treated within categorical settings by \cite{BL} and \cite{DK} respectively. However, there are still plenty of connections left to be made and many further topics within functional analysis where categorical language may provide both motivation to ask new questions and methods to find appropriate answers.

The aim of this paper is to bring together category theory and operator semigroups with the special aim to revisit previous operator theoretic results with fresh eyes. In the theory of \textit{strongly continuous semigroups}, also known as \textit{$C_0$-semigroups}, abstract functional analytic and operator theoretic machinery can be brought to bear upon partial differential equations to tackle questions of well-posedness and asymptotic behaviour (see for example \cite{ABHN,BaDu,BPS,BoTo,EN,NS,RSS}). In all that follows, we denote by $\B(X)$ the space of bounded linear operators on a Banach space $X$, $\N$ the natural numbers, $\Z_+ = \N\cup\{0\}$ the non-negative integers, and $\R_+=[0,\infty)$ the non-negative reals.

\begin{definition}
	A \textit{$C_0$-semigroup} is a map $T:\R_+ \to \B(X)$ which is continuous in the strong operator topology and satisfies the following functional equation:
	\begin{align*}
	T(0) & = I, \\
	T(s+t) & = T(s)T(t), \quad s,t \in \R_+.
	\end{align*}
\end{definition}

If $\R_+$ is replaced by $\R$ in the above definition, then $T$ is a $C_0$-group. Though $C_0$-semigroups can be thought of in various ways, including as solutions to so-called Cauchy problems or as particular inverse Laplace transforms (see \cite[Ch.~3]{ABHN}), we will think of them as representations of the semigroup $\R_+$ in $\B(X)$. This is natural since $T$ is a semigroup homomorphism from the semigroup $\R_+$ under addition to the semigroup $\B(X)$ under composition. Likewise, we will think of $C_0$-groups as representations of the group $\R$ in $\B(X)$. This is an algebraic interpretation that lends itself to categorical formulation as we shall see in Section \ref{sec:3}. We will sometimes refer to $C_0$-semigroups just as semigroups, when it is clear from context what we mean.

Central to our overall aim is the desire to better understand the extension of $C_0$-semigroups to $C_0$-groups. Douglas \cite{D} and others \cite{Breh,cooper,Ito} addressed this in the context of isometric semigroups. Batty and Yeates \cite{BY} generalised Douglas' construction to include more general semigroups, including in particular, $C_0$-semigroups $T$ that are \textit{expansive}, that is, satisfy
$$\|T(t)x\|\ge \|x\|, \quad x\in X,t\ge0.$$

Simple examples of semigroups that are expansive include those that are isometric and those that are generated by multiplication operators $M_a$ on function spaces where the multiplier function $a$ has bounded non-negative real part. A slightly more complicated (though not by much) example is the semigroup $T_q$ on $C_0(\R)$, the space of complex continuous functions on $\R$ vanishing at infinity, given by
$$\left[T_q(t)f\right](s) := e^{\int_{s-t}^s q(r)\,dr}f(s-t), \quad f \in C_0(\R), t\ge0, s \in \R,$$ where $q:\R \to \R_+$ is continuous and vanishes at infinity. $T_q$ can be easily checked to map $C_0(\R)$ into $C_0(\R)$ and to satisfy strong continuity, the semigroup property, and expansiveness (see \cite[Ch.~VI \S1a.]{EN} for a related semigroup).

The approach utilised by Batty and Yeates has various advantages, including the scope of its generality. Nonetheless, the construction we will be primarily working with in this paper comes from Arens \cite{A}, proved in the setting of Banach algebras (see also \cite{Arens2,Arens3,AH,boll,boll2} for related results).
\begin{theorem}[{\cite[Theorem 3.1]{A}}]
	Let $c$ be an element of a commutative normed algebra $A$. Then $c$ has an inverse of norm not exceeding $\gamma^{-1}$ in some algebra $B$ containing $A$ if and only if
	\begin{equation*} \inf_{a \in A} \frac{\|ca\|}{\|a\|} =: \gamma >0.\end{equation*}
\end{theorem}
This theorem can be formulated for bounded linear operators and its proof can also be adapted for generators of $C_0$-semigroups as has been done by Batty and Geyer \cite{BG}. What is more, Arens' method of extension can be applied to $C_0$-semigroups themselves as we shall see in Theorem \ref{3t1}.

Our main results, Theorems \ref{main} and Theorem \ref{adjunction}, state that Arens' extension for $C_0$-semigroups is functorial within an appropriate categorical framework and, furthermore, that it forms an adjunction with the suitably defined forgetful functor. This shows that this extension is more than just an \textit{ad hoc} construction. After this analysis, we make a short excursus into dilation theory, pointing out a possible future direction. As a by-product of viewing $C_0$-semigroups through categorical glasses, we also revisit the work on Banach direct sums\footnote{In the Hilbert space setting the author has elsewhere \cite{N} generalised the notion of direct sums of $C_0$-semigroups, here treated as $UB$-coproducts, to direct integrals of $C_0$-semigroups.} by Lachowicz and Moszy\'nski \cite{LM} through the perspective of coproducts.

\section{Categorical preliminaries}\label{sec:2}

Because this paper is targeted at an operator theoretic audience, we introduce three categorical notions -- adjunctions, universal morphisms, and coproducts for those who might be unfamiliar with them. The definitions below are taken from \cite[Ch.~III,~IV]{MacL} and we assume a basic understanding of what categories, objects, morphisms, and functors are (see \cite[Ch.~I]{MacL}). We will use the following two categories and the canonical functors between them as examples in this section to illuminate the more complicated ideas.

The first example is $\mathbf{Set}$, the category of `small sets'. The condition `small' deals with tricky set theoretic issues which we will not go into, however `small' includes nearly all the usual sets we use (see \cite[p.~12]{MacL}). The objects of $\mathbf{Set}$ are all small sets, its morphisms are all the functions between the sets with the usual composition, and the identity for each object is the standard set identity function. The second example we shall use in conjunction with $\mathbf{Set}$ is $\mathbf{Vct}_K$, the category of vector spaces over a fixed field $K$. The objects of $\mathbf{Vct}_K$ are all the vector spaces over $K$, its morphisms are all the linear transformations between the vector spaces with the usual composition, and the identity for each object is the identity transformation.

Define the object function $E: \ob(\mathbf{Set}) \to \ob(\mathbf{Vct}_K)$ by assigning to each set $X$, the vector space $EX$ with basis $X$. The vectors of $EX$ are the formal finite linear combinations $\sum \lambda_i x_i$ with scalar coefficients $\lambda_i \in K$ and $x_i \in X$ and the vector operations are the obvious ones. Define the morphism function $E$ by sending each morphism $p:X \to Y$ to the linear transformation \begin{align*}Ep: EX & \longrightarrow EY \\ \sum \lambda_i x_i & \longmapsto \sum \lambda_i p(x_i).\end{align*} This defines a functor $E: \mathbf{Set} \to \mathbf{Vct}_K$.

The appropriately named \textit{forgetful functor} $F : \mathbf{Vct}_K \to \mathbf{Set}$ is defined by forgetting all algebraic structure, sending a vector space to the set of all vectors and sending linear transformations to the set maps determined by the tranformations. These two functors relate to one another in a way that will be explored now. Note that for $X\in\mathbf{Set}$, $FEX$ is not the same as $X$, as the formal linear combinations in the former are now elements of their own.

Let us consider any two categories $C$ and $D$ and functors $G: C\to D$ and $H: D \to C$. Then for any objects $c \in C$ and $d\in D$, we can ask the question of when it is the case that there is a bijection $\hom(Gc,d) \cong \hom(c,Hd)$. The eyes of any operator theorist should immediately light up upon seeing this expression due to the similarity with adjoint operators and dual spaces. We formalise this intuitive categorical question of adjoints and duals into the following definition.

\begin{definition}\label{adjdef}
	Let $C$ and $D$ be categories. An \textit{adjunction} from $C$ to $D$ is a triple $\langle G,H,\varphi\rangle$ where $G : C \to D$ and $H: D \to C$ are functors and $\varphi$ is a function which assigns to each pair of objects $c \in C, d \in D$ a bijection
	$$\varphi = \varphi_{c,d} : \hom(Gc,d) \cong \hom(c,Hd)$$ such that for any morphisms $f: c' \to c$ in $C$ and $g: d \to d'$ in $D$, the following diagrams commute:
	\begin{figure}[H]   
		{
			\begin{tikzpicture}[every node/.style={midway}]
			\matrix[column sep={4em,between origins}, row sep={2em}] at (0,0) {
				\node(H1) {$c'$}  ; \\
				\node(H3) {$c$}; \\
			};
			\draw[->] (H1) -- (H3) node[anchor=east] {$f$};
			\end{tikzpicture}
		}
		{
			\begin{tikzpicture}[every node/.style={midway}]
			\matrix[column sep={10em,between origins}, row sep={2em}] at (0,0) {
				\node(H1) {$\hom(Gc,d)$}  ; & \node(H2) {$\hom(c,Hd)$}; \\
				\node(H3) {$\hom(Gc',d)$}; & \node (H4) {$\hom(c',Hd),$};\\
			};
			\draw[->] (H1) -- (H2) node[anchor=south]  {$\varphi$};
			\draw[->] (H3) -- (H4) node[anchor=south] {$\varphi$};
			\draw[->] (H1) -- (H3) node[anchor=east] {$(Gf)^*$};
			\draw[->] (H2) -- (H4) node[anchor=west] {$f^*$};
			\end{tikzpicture}
			
			\begin{tikzpicture}[every node/.style={midway}]
			\matrix[column sep={4em,between origins}, row sep={2em}] at (0,0) {
				\node(H1) {$d$}  ; \\
				\node(H3) {$d'$}; \\
			};
			\draw[->] (H1) -- (H3) node[anchor=east] {$g$};
			\end{tikzpicture}
		}
		{
			\begin{tikzpicture}[every node/.style={midway}]
			\matrix[column sep={10em,between origins}, row sep={2em}] at (0,0) {
				\node(H1) {$\hom(Gc,d)$}  ; & \node(H2) {$\hom(c,Hd)$}; \\
				\node(H3) {$\hom(Gc,d')$}; & \node (H4) {$\hom(c,Hd').$};\\
			};
			\draw[->] (H1) -- (H2) node[anchor=south]  {$\varphi$};
			\draw[->] (H3) -- (H4) node[anchor=south] {$\varphi$};
			\draw[->] (H1) -- (H3) node[anchor=east] {$g_*$};
			\draw[->] (H2) -- (H4) node[anchor=west] {$(Hg)_*$};
			\end{tikzpicture}
		}
	\end{figure} Here $k^*(j) = j \circ k, \; j \in \hom(y,z),$ and $k_*(l) = k \circ l, \; l \in \hom(z,x),$ for any morphism $k: x\to y$. In this case, $G$ is called the left-adjoint of $H$.
\end{definition}

\begin{example}\label{exadj}
	Consider the functors $E : \mathbf{Set} \to \mathbf{Vct}_K$ and $F : \mathbf{Vct}_K \to \mathbf{Set}$ described earlier. Each set function of the form $p:X \to FW$, extends uniquely to a linear transformation $p' : EX \to W$ defined by $p'(\sum \lambda_i x_i) = \sum \lambda_i p(x_i)$. The correspondence $\psi : p\mapsto p'$ has an inverse given by the restriction of $p'$ to $X$, $\varphi : p' \mapsto p'|_X$. Thus, we have a bijection
	$$\varphi = \varphi_{X,W} : \hom(EX,W) \cong \hom(X,FW)$$ defined in the same way for all sets $X$ and vector spaces $W$. Routine calculation shows that the diagrams in Definition \ref{adjdef} commute, providing an adjunction $\langle E,F,\varphi\rangle$ from $\mathbf{Set}$ to $\mathbf{Vct}_K$.
\end{example}

Not only are adjunctions interesting from the theoretical perspective of adjoints and duals, they are also useful because they provide `natural transformations' and universal morphisms, via what is known as Yoneda's Lemma and are ubiquitous in abstract algebra and elsewhere. We will not discuss these natural transformations or Yoneda's Lemma (see \cite[Ch.~III,~IV]{MacL}) but instead we directly define universal morphisms.

\begin{definition}\label{defum}
	Let $C$ and $D$ be categories, $H: D \to C$ a functor, and $c$ an object of $C$. A \textit{universal morphism} from $c$ to $H$ is a pair $\langle a,u\rangle$ consisting of an object $a$ of $D$ and a morphism $u: c \to Ha$ of $C$, such that for every pair $\langle d,f\rangle$ consisting of an object $d$ of $D$ and a morphism $f: c \to Hd$ of $C$, there is a unique morphism $f':a \to d$ of $D$ with $f=(Hf')\circ u.$ In other words, every morphism in $C$ from $c$ to an object in the image of $H$ factors uniquely through the universal morphism $u$, as in the following commutative diagram:
	
	\begin{figure}[H]   
		{
			\begin{tikzpicture}[every node/.style={midway}]
			\matrix[column sep={5em,between origins}, row sep={3em}] at (0,0) {
				\node(H1) {$c$}  ; & \node(H2) {$Ha$}; \\
				& \node (H4) {$Hd$};\\
			};
			\draw[->] (H1) -- (H2) node[anchor=south]  {$u$};
			\draw[->] (H1) -- (H4) node[anchor=north east] {$f$};
			\draw[dashed,->] (H2) -- (H4) node[anchor=west] {$Hf'$};
			\end{tikzpicture}
		}
		{
			\begin{tikzpicture}[every node/.style={midway}]
			\matrix[column sep={5em,between origins}, row sep={3em}] at (0,0) {
				\node(H1) {$a$}  ; \\
				\node(H3) {$d$.};\\
			};
			\draw[dashed,->] (H1) -- (H3) node[anchor=east] {$f'$};
			\end{tikzpicture}
		}
	\end{figure}
\end{definition}

\begin{example}
	Continuing with our examples, $\langle EX, u\rangle$ is a universal morphism from a set $X$ to the forgetful functor $F:\mathbf{Vct}_K \to \mathbf{Set}$ where $u: X \to FEX$ is the map that sends each $x\in X$ to $1\cdot x$ regarded as an element of the set made of vectors of $EX$. This can be seen from the fact that for any other vector space $W$, each function $p:X\to FW$ can be extended to a unique linear transformation $p': EX \to W$ as in Example \ref{exadj} and factors into $p= (Fp') \circ u$ as $x \mapsto 1\cdot x \mapsto 1 \cdot p(x)$. Diagrammatically, the following commutes:
	\begin{figure}[H]   
		{
			\begin{tikzpicture}[every node/.style={midway}]
			\matrix[column sep={5em,between origins}, row sep={3em}] at (0,0) {
				\node(H1) {$X$}  ; & \node(H2) {$FEX$}; \\
				& \node (H4) {$FW$};\\
			};
			\draw[->] (H1) -- (H2) node[anchor=south]  {$u$};
			\draw[->] (H1) -- (H4) node[anchor=north east] {$p$};
			\draw[dashed,->] (H2) -- (H4) node[anchor=west] {$Fp'$};
			\end{tikzpicture}
		}
		{
			\begin{tikzpicture}[every node/.style={midway}]
			\matrix[column sep={5em,between origins}, row sep={3em}] at (0,0) {
				\node(H1) {$EX$}  ; \\
				\node(H3) {$W$.};\\
			};
			\draw[dashed,->] (H1) -- (H3) node[anchor=east] {$p'$};
			\end{tikzpicture}
		}
	\end{figure}
\end{example}

Informally speaking, it may be useful for more complicated examples to think of universal morphisms as a way of uniquely (up to isomorphism) factoring a problem into one component that is always the same (universal) and one component that is easier to deal with\footnote{I actually first got the idea of thinking this way from watching \url{https://www.youtube.com/watch?v=qHuUazkUcnU&ab_channel=mlbaker}, where the notion of tensor products is explained as a standardising tool for turning multilinear maps into linear maps.}. In the terms of Definition \ref{defum}, we may want to understand the morphisms between objects in a less understood category $C$ by describing them in terms of morphisms between objects in a better understood category $D$. If there is a universal morphism, this allows us to factor morphisms in $C$ of the form $f: c \to Hd$ into two parts -- a universal component $u$ that is the same regardless of $f$ and a component $Hf'$ which can be described using a morphism $f'$ between objects in $D$ only. Hence we can use the universal morphism to efficiently describe $f$, a morphism in $C$, in terms of $f'$, a morphism in $D$.

Having introduced universal morphisms, we can now define the coproduct. For any category $C$, let $C^I$ be the category of functors between $C$ and $I$ where $I$ is understood as the discrete category for any set $I$, that is, the category of objects with only identity morphisms. Then the functor category $C^I$ has as its objects the $I$-indexed families $a = \{a_i : i \in I\}$ of objects of $C$. The diagonal functor $\Delta : C \to C^I$ sends each $c \in C$ to the constant family where all $c_i =c$ and all morphisms to $1_c$.

\begin{definition}\label{2d1}
	A universal morphism from an object $a = \{a_i : i \in I\}$ of $C^I$ to the functor $\Delta$ is called a \textit{coproduct diagram}. It consists of an object $\amalg_{i} a_i \in C$ called the \textit{coproduct object} and a collection of morphisms $j_i : a_i \to \amalg_{i} a_i$ called \textit{coproduct injections} such that for any collection of morphisms $k_i : a_i \to d$, there is a unique $h: \amalg_{i} a_i \to d$ with $ k_i = h \circ j_i, \; i \in I$.
\end{definition}

The morphisms $j_i$ are called injections, though these are not required to be `injective' in any way. If $|I| = 2$, we can visualise the universality of the diagram as follows:

\begin{figure}[H] 
	\begin{tikzpicture}[every node/.style={midway}]
	\matrix[column sep={6 em,between origins}, row sep={4em}] at (0,0) {
		\node(H1) {$a_1$} ; & \node(H2) {$a_1 \amalg a_2$} ; & \node(H3) {$a_2$} ; \\
		; & \node(H5) {$d$.}; & ; \\
	};
	\draw[->] (H1) -- (H2) node[anchor=south]  {$j_1$};
	\draw[->] (H3) -- (H2) node[anchor=south] {$j_2$};
	\draw[dashed,->] (H2) -- (H5) node[anchor=west] {$h$};
	\draw[->] (H1) -- (H5) node[anchor=north east] {$k_1$};
	\draw[->] (H3) -- (H5) node[anchor=north west] {$k_2$};
	\end{tikzpicture}
\end{figure}

\begin{example}
	In the case of objects $\{W_i : i\in I\}$ in $\mathbf{Vct}_K$, the coproduct is just the direct sum $\bigoplus_{i\in I} W_i$ consisting of vectors with at most finitely many non-zero components together with the obvious embeddings $j_i : W_i \hookrightarrow \bigoplus_{i\in I} W_i$ given by sending $w_i \in W_i$ to the vector with $w_i$ in the $i$-th component and zeros elsewhere.
\end{example}

We warn that in this paper, the word `natural' will be bandied about in the informal generic mathematical sense, and not in the specific categorical sense.

\section{The category of operator semigroups}\label{sec:3}

In representation theory, the so-called `natural transformations' between representations within the categorical context are intertwining operators (see \cite[Ch.~II \S4]{MacL} for example). In our setting where operator semigroups are considered as representations of $\R_+$, given two Banach spaces $X,Y$ and two semigroups $T$ on $X$, $S$ on $Y$, an intertwining operator between the pairs $(X,T)$, $(Y,S)$ would be an operator $U \in \B(X,Y)$ such that the following diagram commutes for all $t\ge 0$:

\begin{center}
	\begin{tikzpicture}[every node/.style={midway}]
	\matrix[column sep={4em,between origins}, row sep={2em}] at (0,0) {
		\node(H1) {$X$}  ; & \node(H2) {$Y$}; \\
		\node(H3) {$X$}; & \node (H4) {$Y$.};\\
	};
	\draw[->] (H1) -- (H2) node[anchor=south]  {$U$};
	\draw[->] (H3) -- (H4) node[anchor=south] {$U$};
	\draw[->] (H1) -- (H3) node[anchor=east] {$T(t)$};
	\draw[->] (H2) -- (H4) node[anchor=west] {$S(t)$};
	\end{tikzpicture}
\end{center}

We can state this in the following definition.

\begin{definition}
	Let $\mathbf{C_0SG}$ be the category of strongly continuous semigroups of operators on Banach spaces with objects \begin{align*}\ob(\mathbf{C_0SG}) := \{(X,T) : X & \text{ is a Banach space, } \\
	& T:\R_+ \to \B(X) \text{ is a } C_0\text{-semigroup}\}\end{align*} and morphisms $$\hom((X,T),(Y,S)) := \{U \in \B(X,Y) : UT(t) = S(t)U, \; t \in \R_+\}.$$ The identity morphism for an object $(X,T)$ is the identity map $I_{(X,T)}=I \in \B(X)$.
\end{definition}

\begin{remark}\label{2ra}
	It is standard to characterise extensions of operator semigroups and embeddings of Banach spaces within larger Banach spaces through isometric intertwining operators. This will be discussed in the next section. We will also identify two objects $(X,T), (Y,S)$ as the same object if there is an isometric isomorphism $i : X\to Y$ intertwining $T$ and $S$.
\end{remark}

Almost exactly the same definition but with $C_0$-semigroups replaced by $C_0$-groups leads to the definition of the category $\mathbf{C_0G}$ and we can do the same for many more examples such as the category of norm-continuous semigroups. One can also use the same way to define a category of bounded operators where the objects are given by Banach spaces paired with bounded linear operators and the morphisms are again intertwining operators. Indeed, for a single bounded linear operator $L$ on a Banach space $X$, one can form the associated discrete semigroup (or representation) $T_L:\Z_+ \to \B(X)$ of $\Z_+$ by setting $$T_L(n) = L^n, \quad n \in \Z_+.$$ Much of what follows in this paper can also be just as easily done for these discrete representations.

Thinking about $C_0$-semigroups naturally leads to thinking about the generators of these semigroups and so we analogously define a category for them. We refer the reader to \cite[Ch.~II, Definition 1.2]{EN} for the definition of the generator of a strongly continuous operator semigroup. In what follows, $D(A)$ denotes the domain of an (unbounded) operator $A$.

\begin{definition}
	Let $\mathbf{C_0SGG}$ be the category of generators of strongly continuous semigroups of operators on Banach spaces with objects \begin{align*}\ob(\mathbf{C_0SGG}) := \{(X,A) : X & \text{ is a Banach space, } \\
	& A:D(A) \to X \text{ is the generator of a } C_0\text{-semigroup}\}\end{align*} and morphisms $$\hom((X,A),(Y,B)) := \{U \in \B(X,Y) : Ux\in D(B) \text{ and } UAx = BUx, \; x \in D(A)\}.$$ The identity morphism for an object $(X,A)$ is the identity map $I_{(X,A)}=I \in \B(X)$.
\end{definition}

Note that this definition requires that a morphism $U$ from $(X,A)$ to $(Y,B)$ must map the domain of $A$ into the domain of $B$. $\mathbf{C_0SGG}$ can of course be thought of as the subcategory of the category of closed operators defined in similar fashion and we can also define the category of generators of strongly continuous groups of operators in the same way. The link between $\mathbf{C_0SG}$ and $\mathbf{C_0SGG}$ is the obvious one.

\begin{lemma}\label{2l1}
	Let $T,S$ be $C_0$-semigroups on Banach spaces $X,Y$ with generators $A,B$ respectively and let $U \in B(X,Y)$. Then $$UT(t) = S(t)U, 
	\quad t \in \R_+,$$ if and only if for all $x\in D(A)$, $UAx \in D(B)$ and $UAx = BUx.$
\end{lemma}

\begin{proof}
	Suppose $U$ intertwines $T$ and $S$ and let $x \in D(A)$. Then
	\begin{align*}UAx = U\left(\lim_{t\to 0^+}\frac{T(t)x-x}{t}\right) & = \lim_{t \to 0^+}\frac{UT(t)x - Ux}{t} = \lim_{t \to 0^+}\frac{S(t)Ux - Ux}{t} = BUx.\end{align*} The existence of the limit implies that $Ux\in D(B)$. Now assume that $Ux \in D(B)$ and $UAx = BUx$ for all $x \in D(A)$. Simple algebra gives us $$U(\lambda-A)^{-1}x = (\lambda-B)^{-1}Ux, \quad x \in X,$$ for positive $\lambda$ sufficiently large enough to be in the resolvent set of both $A$ and $B$ (\cite[Ch.~II, Theorem  3.8]{EN}). Next, \cite[Ch.~II, Theorem 1.10]{EN} implies that the equation
	\begin{align*}\displaystyle\int_{0}^{\infty} e^{-\lambda t}UT(t)x\, dt = U\displaystyle\int_{0}^{\infty} e^{-\lambda t}T(t)x\, dt & = U(\lambda-A) ^{-1}x \\ & = (\lambda-B) ^{-1}Ux = \displaystyle\int_{0}^{\infty} e^{-\lambda t}S(t)Ux\, dt\end{align*} holds for all $x \in X$ and $\lambda $ sufficiently large, proving the converse by the uniqueness of Laplace transforms (see \cite[Theorem 1.7.3]{ABHN}).
\end{proof}

Let $G:\mathbf{C_0SGG}\to \mathbf{C_0SG}$ be the natural map that takes generators to the semigroups that they generate and intertwining operators to themselves. By Lemma \ref{2l1}, $G$ is clearly functorial. We call $G$ the \textit{generation functor} and it is not hard to see that it is an isomorphism. This allows us to rephrase some of the classical $C_0$-semigroup theorems.


\begin{example} We have the following relations through $G$ and $G^{-1}$.
	\begin{enumerate}[(i)]
		\item By exponentiation, the category of bounded linear operators considered as a subcategory of $\mathbf{C_0SGG}$ is isomorphic to the subcategory of norm-continuous semigroups through the restriction of $G$.
		\item By Stone's theorem, the category of skew-adjoint closed operators on Hilbert spaces considered as a subcategory of $\mathbf{C_0SGG}$ is isomorphic to the subcategory of strongly continuous unitary semigroups (on Hilbert spaces) through the appropriate restriction of $G$.
	\end{enumerate}
\end{example}

If we think of $G$ as a functional calculus like exponentiation, $e^{t-}$, we can pose the question on the existence of a spectral mapping theorem in terms of the following commutative diagram where $\sigma(-)$ denotes taking the spectrum of the argument:

\begin{center}
	\begin{tikzpicture}[every node/.style={midway}]
	\matrix[column sep={10em,between origins}, row sep={4em}] at (0,0) {
		\node(H1) {$A$}  ; & \node(H2) {$T(t) = e^{tA}$}; \\
		\node(H3) {$\sigma(A)$}; & \node (H4) {$\sigma(e^{tA})\setminus\{0\}$.};\\
	};
	\draw[->] (H1) -- (H2) node[anchor=south]  {$G = e^{t-}$};
	\draw[->] (H3) -- (H4) node[anchor=south] {$e^{t-}$};
	\draw[->] (H1) -- (H3) node[anchor=east] {$\sigma(-)$};
	\draw[->] (H2) -- (H4) node[anchor=west] {$\sigma(-)\setminus\{0\}$};
	\end{tikzpicture}
\end{center}

That is, when is it true that $\sigma(e^{tA})\setminus\{0\} = e^{t\sigma(A)} :=  \{e^{t\lambda} : \lambda \in \sigma(A)\}$ (see \cite[Ch.~IV, \S3]{EN})? This also suggests that commutative diagrams might also prove to be a useful tool to study questions pertaining to functional calculus, but such an exploration is beyond the scope of this article.

\section{Extension functor}\label{sec:4}

In this section, we consider Arens' construction suitably adapted as in \cite{BG} to extend operator semigroups. We ask questions of this extension in the categorical setting of operator semigroups. In particular, we show that it is functorial. The construction is in fact for general semigroups with lower bounds, but we will restrict to the case where the semigroup $T$ is expansive.

Though the construction itself is not original, the explicit form of the construction is of fundamental importance for this article and will be referred to in subsequent discussion. Hence we first detail the extension in its own subsection.

\subsection{The Arens construction}\label{const}

The following theorem is essentially just the application of \cite[Theorem 1.1]{BG} to the evaluation of an expansive $C_0$-semigroup at the point $t=1$ (see also \cite[Proposition 5.3]{BG}). In the proof, we include the all-important construction.

\begin{theorem}\label{3t1}
	Let $T$ be an expansive $C_0$-semigroup on the Banach space $X$. There exists a $C_0$-group $\tilde{T}$ on a Banach space $\tilde{X}$ and an isometry $\rho \in B(X,\tilde{X})$ such that $\rho T(t) = \tilde{T}(t)\rho$, $\|\tilde{T}(t)\| = \|T(t)\|$ for $t\ge0$, and $\|\tilde{T}(t)\| \le 1$ for $t<0$.
\end{theorem}

\begin{proof}
	To begin, first define the sequence space $$\ell_1(X) := \{f : \N \to X : \|f\|_1  < \infty\}$$ equipped with the norm $\|\cdot\|_1$ defined by
	$$\|f\|_1 := \sum_{i=1}^{\infty} \|f(i)\|_X.$$  We henceforth drop the subscript $1$ on the norm. Then write $$T_t f := (T(t)f(n))_n, \quad f \in \ell_1(X), t\geq 0,$$ as the coordinatewise multiplication by $T(t)$ on $\ell_1(X)$ for $t\ge 0$. We will construct $\tilde{X}$ as a quotient of $\ell_1(X)$ in such a way that the inverse of $T_1$ will be the right shift $R$ on this space.
	
	To this end, let $$J := \overline{\{(I-T_1R)f : f \in \ell_1(X)\}},$$ which we want to identify as $0$. Let $\tilde{X} := \ell_1(X)/J$ be the Banach space with the quotient norm and define the natural map $\rho: X \to \tilde{X}$ by $\rho(x) = x \mathbf{e_1} + J$ where for any $z \in X$, $z\mathbf{e_1}$ is the sequence $(z,0,0,...)$. Clearly, $\|\rho(x)\| \leq\|x\|$ for all $x\in X$. To show that $\rho$ is an isometric embedding, we need to show the reverse inequality.
	
	Let $x \in X$ and let $f \in \ell_1(X)$ be a sequence with finite support. Then
	\begin{equation}\label{3a}
	\|x\mathbf{e_1} - (I-T_1R)f\| = \|x - f(1)\| + \displaystyle\sum_{n \geq 2}\|f(n) - T_1f(n-1)\|
	\end{equation} where the sum is finite, as its terms are given by the right shift of a sequence of finite support. Due to the expansiveness of $T$, we get
	\begin{equation*}
	\|x\| \leq \|x - f(1)\| + \|f(1)\| \leq \|x - f(1)\| + \|T_1f(1)\|.
	\end{equation*}
	In the same way
	\begin{equation*}
	\|T_1f(1)\| \leq \|T_1f(1) - f(2)\| + \|T_1f(2)\|.
	\end{equation*}
	Thus,
	\begin{equation*}
	\|x\| \leq \| x- f(1)\| +\|T_1f(1) - f(2)\| + \|T_1f(2)\|.
	\end{equation*}
	Iterating this further shows that
	\begin{equation}\label{3b}
	\|x\| \leq \|x\mathbf{e_1} - (I-T_1R)f\|
	\end{equation}
	by (\ref{3a}). Since sequences with finite support are dense in $\ell_1(X)$ and $(I-T_1R)$ is bounded, (\ref{3b}) holds for all $f \in \ell_1(X)$. By the definition of the quotient norm, $\|x\| \leq \|\rho(x)\|$ for all $x\in X$, proving that $\rho$ is an isometric embedding.
	
	Define $\tilde{T} : \R_+ \to \tilde{X}$ by
	$$\tilde{T}(t)(f+J) := T_t f + J, \quad f \in \ell_1(X), t\ge0.$$ $\tilde{T}(t)$ is well defined since $T_t$ commutes with $T_1$ and $R$, leaving $J$ invariant.
	
	Immediately from the definition, $\rho T(t) = \tilde{T}(t)\rho$ and $\|\tilde{T}(t)\| \leq \|T(t)\|$ for all $t\ge0$. Since $\rho$ is an isometry, $\|\tilde{T}(t)\| = \|T(t)\|$ for all $t\ge0$.
	
	Let $f \in \ell_1(X)$ be any sequence with finite support. It follows that $\tilde{T}(t)(f+J) \to f+J$ as $t \to 0$ as $T$ is a $C_0$-semigroup on $X$. Since the set of sequences with finite support is dense in $\ell_1(X)$ and we have uniform boundedness of $\|\tilde{T}(t)\| = \|T(t)\|$ in any compact interval $[0,t_0]$ for $t_0>0$, a standard $\varepsilon/3$ argument shows that $\tilde{T}(t)$ converges strongly to $I$ in $\tilde{X}$. It is clear that $\tilde{T}$ satisfies the semigroup property and hence is a $C_0$-semigroup on $\tilde{X}$.
	
	We now show that $\tilde{T}(1)$ is invertible in $\tilde{X}.$ Define $V$ on $\tilde{X}$ by $V(f+J) := Rf + J$. As $J$ is invariant under $R$, $V$ is well-defined and $\|V\|\leq\|R\|=1$. For $f \in \ell_1(X)$,
	$$f+ J = (I-T_1R)f + T_1Rf + J = T_1Rf +J = \tilde{T}(1)V(f + J)$$ so that $\tilde{T}(1)V = I.$ Likewise, $V\tilde{T}(1) = I$ as $T_1$ commutes with $R$. Hence $\tilde{T}(1)$ is invertible and $\tilde{T}(1)^{-1} = V.$ Since $\tilde{T}$ is a semigroup and $\tilde{T}(1)^{-1}$ exists, it can be extended in the natural way to a $C_0$-group.
\end{proof}

\begin{remark}
	In light of Remark \ref{2ra}, $\tilde{T}$ can be thought of as an extension of $T$ since the intertwining operator $\rho$ is an isometry.
\end{remark}

\subsection{Categorical properties of the extension}

Now that we have the explicit extension, we are ready to prove various properties of interest about it.

Let $\mathbf{C_0SG_{ex}}$ denote the subcategory of $\mathbf{C_0SG}$ where the semigroups of the objects are expansive and let $\mathbf{C_0G_{ex}}$ denote the subcategory of $\mathbf{C_0G}$ where the groups of the objects are expansive on $\R_+$. Define the map 
\begin{align*}
E:\ob(\mathbf{C_0SG_{ex}}) & \to \ob(\mathbf{C_0G_{ex}}) \\
(X,T) & \mapsto (\tilde{X},\tilde{T}),
\end{align*} that is, by assigning every expansive semigroup with its extension constructed in Theorem \ref{3t1}.

Our goal is to further define $E$ so that it not only maps objects to objects, but also morphisms to morphisms. To this end, we prove the following proposition that lifts intertwining operators to the extension spaces.

\begin{proposition}\label{3p1}
	Given $(X,T),(Y,S) \in \mathbf{C_0SG_{ex}}$ and $U \in \hom((X,T),(Y,S))$, let $\rho$ and $\sigma$ denote the isometries constructed in Theorem \ref{3t1} in extending $(X,T),(Y,S)$ to $(\tilde{X},\tilde{T}),(\tilde{Y},\tilde{S})$ respectively. We can find a unique $\tilde{U} \in \hom((\tilde{X},\tilde{T}),(\tilde{Y},\tilde{S}))$ such that the following diagrams commute for all $t\in \R$\emph{:}
	
	\begin{figure}[H]   
		{
			\begin{tikzpicture}[every node/.style={midway}]
			\matrix[column sep={5em,between origins}, row sep={3em}] at (0,0) {
				\node(H1) {$X$}  ; & \node(H2) {$Y$}; \\
				\node(H3) {$\tilde{X}$}; & \node (H4) {$\tilde{Y}$};\\
			};
			\draw[->] (H1) -- (H2) node[anchor=south]  {$U$};
			\draw[->] (H3) -- (H4) node[anchor=south] {$\tilde{U}$};
			\draw[->] (H1) -- (H3) node[anchor=east] {$\rho$};
			\draw[->] (H2) -- (H4) node[anchor=west] {$\sigma$};
			\end{tikzpicture}
		}
		{
			\begin{tikzpicture}[every node/.style={midway}]
			\matrix[column sep={5em,between origins}, row sep={3em}] at (0,0) {
				\node(H1) {$\tilde{X}$}  ; & \node(H2) {$\tilde{Y}$}; \\
				\node(H3) {$\tilde{X}$}; & \node (H4) {$\tilde{Y}$.};\\
			};
			\draw[->] (H1) -- (H2) node[anchor=south]  {$\tilde{U}$};
			\draw[->] (H3) -- (H4) node[anchor=south] {$\tilde{U}$};
			\draw[->] (H1) -- (H3) node[anchor=east] {$\tilde{T}(t)$};
			\draw[->] (H2) -- (H4) node[anchor=west] {$\tilde{S}(t)$};
			\end{tikzpicture}
		}
	\end{figure}
	Moreover, this $\tilde{U}$ satisfies $\|U\|=\|\tilde{U}\|$.
\end{proposition}

\begin{proof}
	Let $K,\mathbf{d_1}$ be the subspace and `element' of $\ell_1(Y)$ analogous to $J,\mathbf{e_1}$ of $\ell_1(X)$ in Theorem~\ref{3t1} respectively. Let $M_U f:= (U f(n))_n \in \ell_1(Y), \; f \in \ell_1(X),$ be the coordinatewise multiplication by $U$. Define $\tilde{U}:\tilde{X}\to\tilde{Y}$ by $$\tilde{U}(f+J) := M_U f + K, \quad f \in \ell_1(X).$$ Now
	$$M_U(I-T_1R)f = M_Uf - S_1M_URf = (I - S_1R)M_Uf, \quad f \in \ell_1(X),$$ as $UT(1) = S(1)U$, mapping $J$ to $K$ so that $\tilde{U}$ is well-defined. Clearly, $\tilde{U} \in \B(\tilde{X},\tilde{Y})$, $\tilde{U}\tilde{T}(t) = \tilde{S}(t)\tilde{U}$ for $t\ge0$, and $\|U\| \geq \|\tilde{U}\|$. Furthermore, for $f\in\ell_1(X)$,
	$$\tilde{U}\tilde{T}(-1)(f+J) = \tilde{U}(Rf+J)= M_U Rf + K = RM_Uf + K = \tilde{S}(-1)U(f+J),$$ so that $\tilde{U}\tilde{T}(t) = \tilde{S}(t)\tilde{U}$ for $t\in\R$ including the negative reals. Since
	$$\tilde{U}\rho(x) = \tilde{U}(x\mathbf{e_1}+J) = Ux\mathbf{d_1}+K = \sigma(Ux),\quad x \in X,$$ the diagrams commute and as $\sigma,\rho$ are isometries, $\|U\| = \|\tilde{U}\|$. It remains to show that $\tilde{U}$ is unique. Let $V_1$ and $V_2$ be two operators that intertwine $\tilde{T}$ and $\tilde{S}$ such that $$V_i \rho = \sigma U, \quad i=1,2.$$ Hence $(V_1 - V_2)\rho = 0$. Note that on $\tilde{X} = \ell_1(X)/J$, $\tilde{T}(-1)$ is defined as the right shift on the quotient space and $\rho(X) = \{x\mathbf{e_1} + J : x \in X\}$. Hence
	\begin{equation}\spn\left\{\label{3e1}\bigcup\limits_{n\geq0}\tilde{T}(-n)\rho(X) \right\}= \{f+J : f \in \ell_1(X) \text{ has finite support}\}\end{equation} and this set is dense in $\tilde{X}.$ Since the $V_i$ intertwine $\tilde{T}$ and $\tilde{S}$, $$(V_1-V_2)\tilde{T}(-n)\rho(x) = \tilde{S}(-n)(V_1 - V_2)\rho(x) = 0, \quad x \in X, n\ge0.$$ By density, $V_1 = V_2$ on $\tilde{X}$.
\end{proof}

\begin{remark}\label{3r1}
	Equation (\ref{3e1}) says that the extension $\tilde{X}$ of $X$ is in fact minimal in the sense that there is no proper closed subspace of $\tilde{X}$ containing $X$ which is invariant under the group $\tilde{T}$. This is equivalent to the subspace $\{\tilde{T}(t)\rho(x) : t \in \R, x \in X\}$ being dense in $\tilde{X}$. This minimality makes the extension somewhat natural, even though the construction itself requires the seemingly arbitrary choice of `transforming the right shift into the inverse of $T(1)$'. This extension is in fact equivalent to the more natural one found in \cite{BY} (see \cite[Remark (3) p~.148]{BY}).
\end{remark}

For each $U \in \hom((X,T),(Y,S))$, define $$EU = \tilde{U} \in \hom(E(X,T),E(Y,S)).$$ Hence $E:\mathbf{C_0SG_{ex}} \to \mathbf{C_0G_{ex}}$ now maps objects to objects and morphisms to morphisms. The following theorem is the central observation of this paper.

\begin{theorem}\label{main}
	$E:\mathbf{C_0SG_{ex}} \to \mathbf{C_0G_{ex}}$ is functorial.
\end{theorem}

\begin{proof}
	Let $U\in \B(X,Y)$ and $V\in\B(Y,Z)$ for Banach spaces $X,Y,Z$. Observe that $M_{UV}f = M_{V}M_{U}f$ for all $f\in \ell_1(X)$ where $M_W$ denotes coordinatewise multiplication for any bounded linear operator $W$ as in the proof of Proposition \ref{3p1}. The result directly follows from the way intertwining operators are extended in Proposition \ref{3p1}.
\end{proof}

We call $E$ the \textit{extension functor}. A similar construction can be found in \cite[Theorem 3.2]{BG} that creates a functor $\bar{E}:\mathbf{C_0SGG_{ex}} \to \mathbf{C_0SGG}$ extending generators with lower bounds to invertible generators and lifting the intertwining operators to the extended space. Here $\mathbf{C_0SGG_{ex}}$ is the subcategory of semigroup generators $A$ that satisfy $$\|Ax\| \ge \|x\|, \quad x\in D(A).$$ However, the role of $\bar{E}$ is signficantly different from that of $E$ since $\bar{E}$ produces an invertible generator whereas $E$ produces an invertible semigroup. Thus, we do not have that $EG = G\bar{E}$ even when restricted to a common domain.

There is a separate closely related functor which naturally arises in the setting of groups and semigroups. This is the \textit{forgetful functor} $F:\mathbf{C_0G}\to \mathbf{C_0SG}$ and it assigns each $C_0$-group $T:\R \to \B(X)$ to the $C_0$-semigroup $T:\R_+ \to \B(X)$ obtained by forgetting $T(t)$ for $t <0.$ The intertwining operators of groups remain intertwining operators between semigroups, and hence $F$ assigns morphisms to themselves. 
\begin{remark}
	If $(X,T) \in\mathbf{C_0SG_{ex}} \subset \mathbf{C_0SG}$, then $\rho:X\to\tilde{X}$ given by the isometric intertwining operator constructed in Theorem \ref{3t1} can also be considered as a morphism in $\hom((X,T),F(\tilde{X},\tilde{T}))$.
\end{remark}
We can clearly restrict $F$ to groups that are expansive on $\R_+$ and from now on, $F$ will refer to the restricted functor
$$F: \mathbf{C_0G_{ex}} \to \mathbf{C_0SG_{ex}}.$$

The question now is in what way $E$ and $F$ relate to each other. It is emphatically not true that for an object $(X,T) \in \mathbf{C_0SG_{ex}}$, $FE(X,T)$ equals $(X,T)$ up to isomorphism. This is because $E$ `embeds' $X$ into a space that can be strictly larger, but $F$ leaves the space itself untouched. However, it is a different story for the other way around.

\begin{lemma}\label{EFI}
	Let $(Y,S) \in \mathbf{C_0G_{ex}}$ so that $F(Y,(S(t))_{t\in \R}) = (Y,(S(t))_{t \in \R_+})$ and let $(\tilde{Y},\tilde{S}) = EF(Y,S)$. Then the associated intertwining isometry $\kappa : Y \to \tilde{Y}$ that embeds $Y$ into $\tilde{Y}$ is an isomorphism.
\end{lemma}

\begin{proof}
	Let $(\bar{S}(t))_{t\in\R}$ be defined on $\kappa(Y) \subset \tilde{Y}$ by
	$$\bar{S}(t)\kappa y := \kappa S(t) y, \quad y \in Y, t \in \R.$$
	By construction, $\tilde{S}(t)|_{\kappa(Y)} = \bar{S}(t)$ for $t \in \R_+$. We will show that this equality holds for $t <0$ as well. Let $x=\kappa y \in \kappa(Y)$. Then
	$$\bar{S}(-t)\tilde{S}(t)x = \bar{S}(-t)\tilde{S}(t)\kappa y = \bar{S}(-t)\kappa S(t)y = \kappa S(-t)S(t)y = \kappa y = x, \quad t \ge 0.$$ Likewise, $\tilde{S}(t)\bar{S}(-t)x = x$ for all $x \in \kappa(Y), t\ge 0$. Note that this only works because $S(t)$ is already defined on $Y$ for $t<0$. Thus
	$$\bar{S}(-t) = \tilde{S}(t)|_{\kappa(Y)}^{-1} = \tilde{S}(-t)|_{\kappa(Y)}, \quad t \ge 0.$$ In particular, as $\bar{S}(t)$ maps $\kappa(Y)$ into $\kappa(Y)$ for all $t \in \R$, $\kappa(Y)$ is a closed $\tilde{S}$ invariant subpsace. By minimality in the sense of Remark \ref{3r1}, $\kappa(Y) = \tilde{Y}$.
\end{proof}

We are now ready to prove the main insight gained by rethinking the Arens' extension in fresh categorical language. In what follows, we identify $EF(Y,S)$ with $(Y,S)$ via the previous lemma and suppress the map $\kappa$ where it appears implicitly.

\begin{theorem}\label{adjunction}
	$\langle E, F, \varphi\rangle$ is an adjunction from $\mathbf{C_0SG_{ex}}$ to $\mathbf{C_0G_{ex}}$ where for any $(X,T) \in \mathbf{C_0SG_{ex}}$ and $(Y,S)\in \mathbf{C_0G_{ex}}$, $\varphi$ is the bijection given by
	\begin{align*} \varphi : \hom(E(X,T),(Y,S)) & \cong \hom((X,T),F(Y,S))\\
	U & \longmapsto (FU)|_{X} = U|_X \\
	EV & \longmapsfrom V.
	\end{align*} Thus $E$ is the left-adjoint of $F$.
\end{theorem}

\begin{proof}
	Let $(X,T),(X',T') \in \mathbf{C_0SG_{ex}}$ and $(Y,S),(Y',S') \in \mathbf{C_0G_{ex}}$ and let $U \in \hom((X',T'),(X,T))$ and $V \in \hom((Y,S),(Y',S'))$. Denote $\varphi^{-1}$ by $\psi$. Then for any $W \in \hom((X,T),F(Y,S))$ we have
	$$\psi(U^*(W)) = E(WU) = E(W)E(U) = (EU)^*(EW) = (EU)^*(\psi(W))$$
	and
	$$\psi((FV)_*(W)) = E((FV)W) = (EFV)(EW) = V(EW) = V_*(\psi(W))$$ where $EFV=V$ by Lemma \ref{EFI}. Thus the following diagrams commute:
	\begin{figure}[H]   
		{
			\begin{tikzpicture}[every node/.style={midway}]
			\matrix[column sep={4em,between origins}, row sep={2em}] at (0,0) {
				\node(H1) {$(X',T')$}  ; \\
				\node(H3) {$(X,T)$}; \\
			};
			\draw[->] (H1) -- (H3) node[anchor=east] {$U$};
			\end{tikzpicture}
		}
		{
			\begin{tikzpicture}[every node/.style={midway}]
			\matrix[column sep={13em,between origins}, row sep={2em}] at (0,0) {
				\node(H1) {$\hom((X,T),F(Y,S))$}  ; & \node(H2) {$\hom(E(X,T),(Y,S))$}; \\
				\node(H3) {$\hom((X',T'),F(Y,S))$}; & \node (H4) {$\hom(E(X',T'),(Y,S)),$};\\
			};
			\draw[->] (H1) -- (H2) node[anchor=south]  {$\psi$};
			\draw[->] (H3) -- (H4) node[anchor=south] {$\psi$};
			\draw[->] (H1) -- (H3) node[anchor=east] {$U^*$};
			\draw[->] (H2) -- (H4) node[anchor=west] {$(EU)^*$};
			\end{tikzpicture}
			
			\begin{tikzpicture}[every node/.style={midway}]
			\matrix[column sep={4em,between origins}, row sep={2em}] at (0,0) {
				\node(H1) {$(Y,S)$}  ; \\
				\node(H3) {$(Y',S')$}; \\
			};
			\draw[->] (H1) -- (H3) node[anchor=east] {$V$};
			\end{tikzpicture}
		}
		{
			\begin{tikzpicture}[every node/.style={midway}]
			\matrix[column sep={13em,between origins}, row sep={2em}] at (0,0) {
				\node(H1) {$\hom((X,T),F(Y,S))$}  ; & \node(H2) {$\hom(E(X,T),(Y,S))$}; \\
				\node(H3) {$\hom((X,T),F(Y',S'))$}; & \node (H4) {$\hom(E(X,T),(Y',S')).$};\\
			};
			\draw[->] (H1) -- (H2) node[anchor=south]  {$\psi$};
			\draw[->] (H3) -- (H4) node[anchor=south] {$\psi$};
			\draw[->] (H1) -- (H3) node[anchor=east] {$(FV)_*$};
			\draw[->] (H2) -- (H4) node[anchor=west] {$V_*$};
			\end{tikzpicture}
		}
	\end{figure}
	In particular, $\varphi = \psi^{-1}$ satisfies the requirements of Definition \ref{adjdef}.
\end{proof}

As mentioned in Section \ref{sec:2}, adjunctions immediately imply the existence of universal morphisms via Yoneda's Lemma. However, in our case, we can show this directly.

\begin{theorem}\label{3t2}
	Consider the forgetful functor $F:\mathbf{C_0G_{ex}}\to \mathbf{C_0SG_{ex}}$ and $(X,T) \in \mathbf{C_0SG_{ex}}$. Then $\langle E(X,T),\rho\rangle$ is a universal morphism from $(X,T)$ to $F$, where $\rho$ is the isometric intertwining operator embedding $X$ into its extension.
\end{theorem}

\begin{proof}
	Let $(Y,S) \in \mathbf{C_0G_{ex}}$ and $U \in \hom((X,T),F(Y,S))$. Denote $E(X,T)$ by $(\tilde{X},\tilde{T})$. Let $\kappa, I,$ and $J$ be the morphisms given by the isometric isomorphism in Lemma \ref{EFI}, and the identity operators on $\tilde{X}$ and $\kappa(Y)$ respectively. Then by Proposition \ref{3p1}, there exists a unique $\tilde{U} = EU$ such that the following diagram commutes:
	
	\begin{figure}[H]   
		{
			\begin{tikzpicture}[every node/.style={midway}]
			\matrix[column sep={11em,between origins}, row sep={4em}] at (0,0) {
				\node(H1) {$(X,T)$}  ; & \node(H2) {$(\tilde{X},\tilde{T})$}; & \node(H3) {$F(\tilde{X},\tilde{T})$}; \\
				\node(H4) {$F(Y,S)$} ; & \node(H5) {$EF(Y,S)$}; & \node(H6) {$FEF(Y,S).$} ;\\
			};
			\draw[->] (H1) -- (H2) node[anchor=south]  {$\rho$};
			\draw[->] (H4) -- (H5) node[anchor=south] {$\kappa$};
			\draw[->] (H1) -- (H4) node[anchor=east] {$U$};
			\draw[->] (H2) -- (H5) node[anchor=east] {$\tilde{U}$};
			\draw[->] (H2) -- (H3) node[anchor=south] {$I$};
			\draw[->] (H5) -- (H6) node[anchor=south] {$J$};
			\draw[->] (H3) -- (H6) node[anchor=east] {$F\tilde{U}$};
			\end{tikzpicture}
		}
	\end{figure}
	In particular, $\tilde{U}$ is the unique morphism such that the following rearranged diagram commutes after identifying $EF(Y,S)$ with $(Y,S)$ by Lemma \ref{EFI} and writing $I\circ \rho = \rho$ since they are equal as intertwining operators from $X$ to $\tilde{X}$:
	\begin{figure}[H]   
		{
			\begin{tikzpicture}[every node/.style={midway}]
			\matrix[column sep={7em,between origins}, row sep={4em}] at (0,0) {
				\node(H1) {$(X,T)$}  ; & \node(H2) {$FE(X,T)$}; \\
				& \node (H4) {$F(Y,S)$};\\
			};
			\draw[->] (H1) -- (H2) node[anchor=south]  {$\rho$};
			\draw[->] (H1) -- (H4) node[anchor=north east] {$U$};
			\draw[dashed,->] (H2) -- (H4) node[anchor=west] {$F\tilde{U}$};
			\end{tikzpicture}
		}
		{
			\begin{tikzpicture}[every node/.style={midway}]
			\matrix[column sep={7em,between origins}, row sep={4em}] at (0,0) {
				\node(H1) {$E(X,T)$}  ; \\
				\node(H3) {$(Y,S)$.};\\
			};
			\draw[dashed,->] (H1) -- (H3) node[anchor=east] {$\tilde{U}$};
			\end{tikzpicture}
		}
	\end{figure}
\end{proof}

The extension $E(X,T) = (\tilde{X},\tilde{T})$ via $\rho:X\to\tilde{X}$ is also universal in the stronger sense of \cite[\S 3]{BY}, that is, for all minimal extensions $(Y,S) \in \mathbf{C_0G_{ex}}$ with intertwining isometry $\pi$,
\begin{equation*}\label{3e4}\|S(t)\pi(x)\| \leq \|\tilde{T}(t)\rho(x)\|, \quad x\in X, t\in \R.\end{equation*} As in \cite[\S 3]{BY}, this is equivalent to the existence of a (unique) linear contraction $\lambda : \tilde{X} \to Y$ such that $\lambda \circ \rho = \pi$ and $S(t) \lambda = \lambda\tilde{T}(t), \; t\in\R$. To see that we have such a $\lambda$, suppose that $(Y,S)$ is another minimal extension via $\pi:X\to Y$ of $(X,T)$. Then $\pi$ can be considered a morphism in $\hom((X,T),F(Y,S))$. By Theorem \ref{3t2},	we can uniquely factor $\pi$ into $\pi = \tilde{\pi}\circ \rho$ for some $\tilde{\pi} : \tilde{X} \to Y$ since $F\tilde{\pi}$ is the same as $\tilde{\pi}$ as an operator. Pictorially, all we are doing is replacing a general morphism $U$ with the intertwining isometry $\pi$ in the first part of the first diagram of the previous proof as follows:
\begin{figure}[H]   
	{
		\begin{tikzpicture}[every node/.style={midway}]
		\matrix[column sep={11em,between origins}, row sep={4em}] at (0,0) {
			\node(H1) {$(X,T)$}  ; & \node(H2) {$(\tilde{X},\tilde{T})$}; \\
			\node(H4) {$F(Y,S)$} ; & \node(H5) {$E(F(Y,S)) = (Y,S).$}; \\
		};
		\draw[->] (H1) -- (H2) node[anchor=south]  {$\rho$};
		\draw[->] (H4) -- (H5) node[anchor=south] {$\kappa$};
		\draw[->] (H1) -- (H4) node[anchor=east] {$\pi$};
		\draw[->] (H2) -- (H5) node[anchor=east] {$\tilde{\pi}$};
		\end{tikzpicture}
	}
\end{figure}
As $\tilde{\pi}$ is here considered a morphism in $\hom((\tilde{X},\tilde{T}),(Y,S))$, this precisely means that $S(t)\tilde{\pi} = \tilde{\pi}T(t), \; t \in \R$. Finally, as $\tilde{\pi}$ is constructed via Proposition \ref{3p1}, $\|\tilde{\pi}\| = \|\pi\| = 1$. Thus $\tilde{\pi}$ is the linear contraction $\lambda$ we were after.

As stated in Remark \ref{3r1}, this also shows that the Batty-Yeates extension is identical to the extension by $E$ in the sense of Remark \ref{2ra} by \cite[Theorem 3.3]{BY}. The Batty-Yeates extension is constructed in a less artificial way, as well as for a more general setting. The practical advantage of $E$ is that the construction is much more explicit and provides the norm of the extended space in a form that is much easier to use than the norm given in \cite{BY}. The theoretical advantage of $E$ is that the explicit form of the extension allows us to lift intertwining operators as in Proposition \ref{3p1}.

\subsection{Comments on additional properties of the space or semigroup}

Though not the main point of this article, which is to provide a higher level view of the extension construction, some words are perhaps in order regarding additional properties of either the relevant spaces or $C_0$-semigroups.

First, the only assumption underlying everything done so far is that $X$ is a Banach space. Naturally, one is moved to ask about the case where $X$ is Hilbert. In answer to that, the construction in Subsection \ref{const} clearly destroys the Hilbert space structure, given that the sequence space $\ell_1(X)$ is used. What if an analogous $\ell_2(X)$ space is used instead? This has been previously addressed in the setting of a single bounded linear operator (and hence also for discrete operator semigroups) by \cite[Theorem 6.3]{BG} where the result \cite[Corollary 4.8]{BM} is restated. However, in that case, the newly created inverses are no longer bounded by $1$ as in Theorem \ref{3t1}. This led Batty and Geyer to subsequently pose \cite[Open Question 6.4]{BG}, for which we are currently unaware of any resolution.

It is worth noting here that Read also produced a different extension that removed the residual spectrum for single linear bounded operators on Banach spaces \cite{Read1,Read3}. In the construction of an analogous extension for Hilbert spaces \cite{Read2}, the norms of the elements in the commutant of the operator being extended were only `almost' preserved. It remains an open question whether there is a spectrum-reducing extension on Hilbert spaces that completely preserves the norms of elements in the commutant. All this leads the author to believe that a structure preserving extension theory for Hilbert spaces as comprehensive and nice as what we were able to do for general Banach spaces may be difficult to find.

Second, for $(\tilde{X},\tilde{T}) = E(X,T)$ where $(X,T) \in \mathbf{C_0SG_{ex}}$, we know that $\|\tilde{T}(-1)\| \le 1$. However, is this inequality in fact equality if $T$ is a semigroup of isometries? The answer is yes, though the author finds it hard to show this directly from the construction. Instead, we show that the Arens extension is equivalent to the extension Douglas \cite{D} already constructed for isometric semigroups as mentioned in the introduction. We do so as follows.

Let $T$ be a semigroup of isometries on a Banach space $X$ and let $(X_D,T_D)$ denote the Douglas extension to a $C_0$-group. Then $(X_D,T_D)$ is a minimal extension that is also universal in the stronger sense of \cite[\S 3]{BY} (see \cite[Example 3.1]{BY}). Let $(\tilde{X},\tilde{T}) = E(X,T)$ and denote by $(X_{BY},T_{BY})$ the Batty--Yeates extension. By \cite[Theorem 3.3]{BY}, there exist isometric isomorphisms $i:X_{BY} \to \tilde{X}$ and $ j:X_{BY} \to X_D$ such that $i \circ T_{BY} = \tilde{T}  \circ i$ and $  j \circ T_{BY} = T_D \circ j$. So $(\tilde{X},\tilde{T})$ and $(X_D,T_D)$ can be identified together via the isometric isomorphism $j\circ i^{-1}$. In particular,
$$\|\tilde{T}(t)x\| = \|j\circ i^{-1} \circ \tilde{T}(t)x\| = \| T_D(t) \circ j \circ i^{-1} x\| = \|x\|, \quad  x \in \tilde{X}, t\in \R.$$
Note that Douglas' construction also preserves Hilbert space structure \cite[Theorem 1]{D}, which is interesting given our first point of discussion in this subsection about the difficulties that arise when working with Hilbert spaces.

Third, as seen in the mention of spectral mapping theorems at the end of Section \ref{sec:3}, questions regarding the spectrum of both generator and semigroup are generally of interest. Here, the Read extension mentioned above is more relevant, as it specifically deals with spectrum. Nonetheless, we mention briefly that for our extension $E(X,T) = (\tilde{X},\tilde{T})$, \cite[Proposition 2.2]{BG} applies, and furthermore the space $Y_0$ in \cite[Proposition 2.2]{BG} is the whole space $\tilde{X}$. Thus $\sigma(G^{-1}FE(X,T))\subset \sigma(G^{-1}(X,T))$, where $\sigma(-)$ is understood to denote the spectrum of the operator component of the objects in $\mathbf{C_0SGG}$.

There remain many interesting questions regarding the Arens extension. For example, does it preserve lattice structure? Positivity of the semigroup? Though these questions fall outside of our present scope, we hope the reader will find it stimulating to think about these things.

\section{Excursus -- Isometric dilations of Hilbert space contractions}\label{sec:5}

This short section following our discussion on extensions is a slight digression where we raise the possibility of future approaches to the dilation of operator semigroups via a similar categorical framework. The idea to address this possibility was suggested by the reviewer.

It is well known (the classic text is \cite{BFKSZ}) that for any linear contraction $T$ on a Hilbert space $H$, there exist a Hilbert space $\tilde{H} \supset H$ and an isometry $\tilde{T}$ on $\tilde{H}$ such that
$$T^n x = P\tilde{T}^n x, \quad x \in H, n \in \Z_+,$$ where $P$ is the orthogonal projection from $\tilde{H}$ onto $H$.

We can actually choose the dilation to be a unitary operator and, furthermore, these dilations can also be found for strongly continuous semigroups. For the sake of this discussion, we will stick to isometric dilations of single contractions (and hence discrete contraction semigroups).

The enveloping space $\tilde{H}$ can be chosen to be minimal in the sense that the smallest invariant subspace for $\tilde{T}$ that contains $H$ is $\tilde{H}$. In general, $\tilde{T}$ is called an isometric \textit{dilation} of $T$ and when $\tilde{H}$ and $\tilde{T}$ are chosen to be minimal, then the dilation is unique up to the existence of an intertwining unitary operator.

Similar to what was done in Section \ref{sec:3}, we can define $\mathbf{CH}$ to be the category of contractive operators on Hilbert spaces consisting of objects $(H,T)$ where $H$ is a Hilbert space and $T$ is a contraction on $H$ and morphisms given by intertwining operators. Let $\mathbf{IH}$ be the category of isometric operators on Hilbert spaces defined in the same way.

The pertinent issue is now whether the object map $D: \ob{(\mathbf{CH})} \to \ob{(\mathbf{IH})}$ that takes a contraction $T$ on a Hilbert space $H$ to its minimal isometric dilation $\tilde{T}$ on $\tilde{H}$ can be further defined so as to be a functor. The most natural way to go about this is via the question of dilating intertwining operators. This question is so involved, however, that there was a series of eight papers written on this subject appropriately titled `On intertwining dilations I--VIII' by Ceau\c{s}escu \textit{et al}. In that extended discourse, the dilations of interest were known as \textit{exact intertwining dilations} or EID, which, for a given intertwining operator $U$ from $(H,T)$ to $(K,S)$, are defined to be the intertwining operators $V$ from $(\tilde{H},\tilde{T})$ to $(\tilde{K},\tilde{S})$ such that $\|U\| = \|V\|$ and the following diagrams commute:

\begin{figure}[H]   
	{
		\begin{tikzpicture}[every node/.style={midway}]
		\matrix[column sep={5em,between origins}, row sep={3em}] at (0,0) {
			\node(H1) {$\tilde{H}$}  ; & \node(H2) {$\tilde{K}$}; \\
			\node(H3) {$H$}; & \node (H4) {$K$};\\
		};
		\draw[->] (H1) -- (H2) node[anchor=south]  {$V$};
		\draw[->] (H3) -- (H4) node[anchor=south] {$U$};
		\draw[->] (H1) -- (H3) node[anchor=east] {$P_H$};
		\draw[->] (H2) -- (H4) node[anchor=west] {$P_K$};
		\end{tikzpicture}
	}
	{
		\begin{tikzpicture}[every node/.style={midway}]
		\matrix[column sep={5em,between origins}, row sep={3em}] at (0,0) {
			\node(H1) {$\tilde{H}$}  ; & \node(H2) {$\tilde{K}$}; \\
			\node(H3) {$\tilde{H}$}; & \node (H4) {$\tilde{K}$.};\\
		};
		\draw[->] (H1) -- (H2) node[anchor=south]  {$V$};
		\draw[->] (H3) -- (H4) node[anchor=south] {$V$};
		\draw[->] (H1) -- (H3) node[anchor=east] {$\tilde{T}$};
		\draw[->] (H2) -- (H4) node[anchor=west] {$\tilde{S}$};
		\end{tikzpicture}
	}
\end{figure}
\noindent Here $P_H$ and $P_K$ are the orthogonal projections from $\tilde{H}$ and $\tilde{K}$ onto $H$ and $K$ respectively.

Prior to this series of papers, it was already known (see \cite{BFKSZ}) that there always exists at least one EID. However there can be more than one, so that the choice of which EID to assign as $DU$ can be ambiguous. Nonetheless, under certain regularity conditions, $U$ does indeed have a unique EID \cite{ACF}, allowing an unequivocal choice for $DU$. Thus, a proper theory similar to what we have produced for the extension of expansive semigroups may be possible, but substantial work befitting a separate article needs to be done.

All this is to say that the theory of dilations of operators and operator semigroups can benefit from more systematic category theoretic approaches. Indeed category theory was already used as a tool to understand Hilbert spaces and dilations in Arveson's survey \cite{Arv}. We note also that category theory really comes into its own when used to unify theories from different areas of mathematics. Since dilation theory is precisely a theory with many counterparts across mathematics -- the Sz.-Nagy theorem for contraction operators on Hilbert spaces, the Akcoglu-Sucheston theorem for positive contractions on reflexive $L^p$-spaces, the Naimark theorem for positive operator-valued measures, and the Stinespring theorem for completely positive maps on $C^*$-algebras -- we suggest that dilation theory may well be fertile ground for future categorical explorations. We note that an abstract structure theoretic approach to dilations of operators has already been developed in the aptly named \cite{FG}.

\section{UB-Coproducts}\label{sec:6}

We now turn our attention to coproducts in order to revisit the work of \cite{LM} through our categorical lens. For the category of vector spaces over a field $K$, the coproduct is precisely the direct sum of elements in the product with finitely many nonzero components. If we applied the same approach to Banach spaces, taking the indexing set $I = \N$, the $\ell_1$ norm would make sense. However, the direct sum of Banach spaces in this algebraic way is not complete. Hence the most ideal situation one could hope for is if the coproduct object of a sequence of Banach spaces $\chi = \{X_i : i \in \N \}$ is an $\ell_p$ space
$$\ell_p(\N,\chi) := \{f \in \ell(\N,\chi) : \|f\|_p  < \infty\}$$ where $\ell(\N,\chi)$ is the vector space of sequences $\{f_i\}_{i\in \N}$ with entries $f_i \in X_i$ and
$$\|f\|_p := \left(\sum_{i=1}^{\infty} \|f_i\|_{X_i}^p\right)^{1/p}$$ for $1\leq p <\infty$ together with component-wise operations.
Clearly $\ell_p(\N,\chi)$ is a Banach space and the sequences of finite support are dense in it. For our purposes, we will take $p=1$. Now given a sequence of Banach spaces $\{X_i\}_{i \in\N}$ and a sequence of operators $\{A_i\}_{i\in\N}$ on these spaces, we can define a diagonal operator on $\ell_1(\N,\chi)$.

\begin{definition}
	The diagonal of the sequence of linear operators $\{A_i : D(A_i) \subseteq X_i \to X_i\}_{i \in \N}$ is the operator $A: D(A) \to X = \ell_1(\N,\chi)$ given by
	\begin{align*}
	D(A) & := \{f \in X : f_i \in D(A_i),\{A_i f_i\}_{i\in \N}\in X\}, \\
	Af & := \{A_i f_i\}_{i\in \N}.
	\end{align*}
	We denote $A$ by $\diag_{i\in\N} A_i$ and say that a linear operator $T$ on $X$ is \textit{diagonal} if and only if $T = \diag_{i\in\N} T_i$ for some sequence $\{T_i\}_{i \in \N}$ of operators.
\end{definition}
We refer to \cite{LM} for a more comprehensive discussion on the so-called Banach direct sums and diagonal operators. The goal is now to characterise the coproducts of semigroups. Given a sequence $(X_i,T_i)$ of objects in $\mathbf{C_0SG}$, the naive construction of the coproduct would be to take $\amalg_{i \in \N}(X_i,T_i) = (\ell_1(\N,\chi),T)$ where $\chi = \{X_i\}_{i\in\N}$ and $T(t) =  \diag_{i\in \N} T_i(t).$ However there are two issues with doing this construction in general, the first being whether or not $T$ is a $C_0$-semigroup at all and the second being the unique factorisation of maps in the universality condition of coproduct diagrams.

First note that if $T$ is a $C_0$-semigroup, then there exist $M > 0$ and $\omega \in \R$ so that  $\|T(t)\| \leq Me^{\omega t}$ for all $t\geq 0$.  Hence, if $T$ is defined at each $t \in \R_+$ by $T(t) =\diag_{i\in\N} T_i(t)$, then taking the elements of $\ell_1(\N,\chi)$ that are $0$ everywhere except in the $i$-th component, we get that $\|T_i(t)\| \leq M e^{\omega t}$ for all $i\in\N$. Thus a uniform exponential growth bound is a necessary condition. For this reason, let us define subcategories of $\mathbf{C_0SG}$ by introducing parameters. Let $\mathbf{C_0SG}(M,\omega)$ denote the category of all $C_0$-semigroups with exponential growth bound $Me^{\omega t}$. Now \cite[Theorem 4.3]{LM} tells us that if there exists $M>0$ and $\omega \in \R$ such that $(X_i,T_i) \in \mathbf{C_0SG}(M,\omega), \; i\in \N,$ then $T(t) =  \diag_{i\in\N} T_i(t)$ is a $C_0$-semigroup on the space $X= \ell_1(\N,\chi)$. This deals with the first issue.

If this diagonal $T$ were to be the coproduct of $\{(X_i,T_i)\}_{i\in\N},$ then the natural injection morphisms would be $$j_i : X_i \to X, \ \ x \mapsto x \mathbf{e_i},$$ where $\mathbf{e_i}$ is the formal representation for the sequence with $1$ in the $i$-th component and $0$ elsewhere. 

Following this line of thought, if $k_i : X_i \to Y$ defined intertwining morphisms in $\hom((X_i,T_i),(Y,S))$ for some $(Y,S) \in \mathbf{C_0SG}$, the naturally induced factorising map $K:X \to Y$ for which $$k_i = K \circ j_i, \quad i \in \N,$$ holds would be given by $$K:f \mapsto \displaystyle\sum_{i \in \N} k_i f_i$$ for $f$ with finite support and extending by density provided this $K$ is bounded. Since \begin{align*}KT(t)f = K(\{T_i(t)f_i\}_{i \in \N}) & = \sum \left(k_i T_i(t)f_i\right)\\ & = \sum \left(S(t)k_i f_i\right) =S(t)\left(\sum k_i f_i \right) = S(t)Kf\end{align*} for all finitely supported $f$, such a $K$ would qualify as a morphism between $(X,T)$ and $(Y,S)$. In order to get this boundedness for $K$ and hence resolve the second issue mentioned earlier, we must restrict to satisfying the universal property for uniformly bounded $k_i$ by adjusting Definition \ref{2d1}.

\begin{definition}\label{defub}
	Given a collection of objects $\{(X_i,T_i)\}_{i \in \N}$ in $\mathbf{C_0SG}$, a \textit{UB-coproduct} is an object $\amalg_{i}^{UB} (X_i,T_i) \in \mathbf{C_0SG}$ called the \textit{UB-coproduct object} together with a collection of intertwining operators (morphisms) $j_i : (X_i,T_i) \to \amalg_{i}^{UB} (X_i,T_i)$ called \textit{UB-coproduct injections} such that for any collection of uniformly bounded intertwining operators (morphisms) $k_i : (X_i,T_i) \to (Y,S)$ for some fixed $(Y,S)\in\mathbf{C_0SG}$, there is a unique morphism $h: \amalg_{i}^{UB} (X_i,T_i) \to (Y,S)$ with $ k_i = h \circ j_i$ for all $i \in \N$.
\end{definition}

We can then restate \cite[Theorem 4.3]{LM} in the following way with an extra clause of uniqueness.

\begin{theorem}\label{ub}
	Let $\{(X_i,T_i)\}_{i\in \N}$ be a collection in $\mathbf{C_0SG}(M,\omega)$. Then the UB-coproduct exists. The $UB$-coproduct object is given by $\amalg_{i\in \N}^{UB}(X_i,T_i) = (\ell_1(\N,\chi),T)$ where $\chi = \{X_i\}_{i\in\N}$ and $T(t) = \diag_{i\in\N} T_i(t)$, the $UB$-coproduct injections $j_i : (X_i,T_i) \to (\ell_1(\N,\chi),T)$ are given by componentwise embeddings, and the unique factorising map is given by the construction for $K$ in the preceding discussion.
\end{theorem}

\begin{proof}
	After applying \cite[Theorem 4.3]{LM}, it remains to show that given an object $(Y,S) \in \mathbf{C_0SG}$, the morphism $K: (\ell_1(\N,\chi),T) \to (Y,S)$ defined via the intertwining operator $K$ in the preceding discussion is unique. Suppose there exists another $g: (\ell_1(\N,\chi),T) \to (Y,S)$ such that $k_i = g \circ j_i$ for all $i\in \N$. Then $(K-g)\circ j_i = 0$ for all $i \in \N$, and in particular $(K-g)f = 0$ for all $f$ with finite support and by density, for all $f \in \ell_1(\N,\chi).$
\end{proof}

\providecommand{\bysame}{\leavevmode\hbox to3em{\hrulefill}\thinspace}
\providecommand{\MR}{\relax\ifhmode\unskip\space\fi MR }
\providecommand{\MRhref}[2]{%
	\href{http://www.ams.org/mathscinet-getitem?mr=#1}{#2}
}
\providecommand{\href}[2]{#2}

\end{document}